\documentclass[12pt, reqno]{amsart}
\usepackage{amsmath, amsthm, amscd, amsfonts, amssymb, graphicx, color}

\newtheorem{df}{Definition}[section]
\newtheorem{thm}[df]{Theorem}
\newtheorem{pro}[df]{Proposition}

\newtheorem{rema}[df] {Remark}

\begin{document}
\setcounter{page}{1}

\title[solvable Lie algebras of operators]{A spectral theory for solvable\\ Lie algebras of operators}
\author{Enrico Boasso and Angel Larotonda}
\begin{abstract} The main objective of this paper is to develop a notion of joint spectrum for
complex solvable Lie algebras of operators acting on a Banach space, which 
generalizes the Taylor joint spectrum (T.J.S.) for several commuting operators. 
\end{abstract}
\maketitle

\section{ Introduction}
\indent We briefly recall the definition of the Taylor joint spectrum. 
Let $\wedge (\Bbb C^n)$ be 
the complex exterior algebra on $n$ generators $e_1,\ldots ,e_n$ with
multiplication denoted by $\wedge$. Let $E$ be a Banach space and $a=
(a_1,\ldots ,a_n)$ a mutually commnuting $n$-tuple of bounded linear operators on E
(m.c.o.). Define $\wedge_k^n (E)=\wedge_k (\Bbb C^n )\otimes_{\Bbb C}E$ and $D_{k-1}$ by:
$$
D_{k-1}\colon \wedge_k^n (E)\to \wedge_{k-1}^n (E)
$$
$$
D_{k-1}(x\otimes e_{i_1}\wedge \cdots \wedge e_{i_k})=\sum_{j=1}^k (-1)^{j+1} x\cdot a_{i_j}
\otimes e_{i_1}\wedge \cdots\wedge \tilde{e}_{i_j}\wedge\cdots\wedge e_{i_k},
$$
where $\tilde{}$ means deletion and 
$k\ge 1$. Also define $D_k=0$ for $k\le 0$.\par

\indent It is easy to prove that $D_kD_{k+1}=0$ for all $k$, that is, $(\wedge_k^n (E), 
D_k)_{k\in\Bbb Z}$ is a chain complex, which is called the Koszul complex associated 
to $a$ and $E$ and it is denoted by $R(E,a)$. The $n$-tuple $a$ is said to be invertible or 
nonsingular on $E$, if $R(E,a)$ is exact, i.e., $Ker (D_k)=Ran (E_{k+1})$ for all $k$.
The Taylor spectrum of $a$ on $E$ is $Sp(a,E)=\{\lambda\in\Bbb C^n: a-\lambda
\hbox{ is not invertible}\}$.\par
\indent Unfortunately, this definition depends very strongly on $a_1,\ldots ,a_n$ and not 
on the vector subspace of $L(E)$ generated by then, i.e.,
$\langle a\rangle$.\par
\indent As we consider Lie algebras, which involve naturally geometry, we are 
interested in a geometrical approach to the spectrum which depends on $L$  rather than on 
a particular set of operators.\par
\indent This is done in section 2. Given a solvable Lie subalgebra of $L(E)$, $L$, we associate to 
it a set in $L^*$, $Sp(L,E)$.\par
\indent This object has the classical properties, i.e., $Sp(L,E)$ is a compact set, if $L{'}$ is
an ideal of $L$, then $Sp(L{'},E)$ is the projection of $Sp(L,E)$ on 
$L{'}^*$, and finally, $Sp(L,E)$ is a non empty set.\par
\indent Besides, it satisfy other interesting properties.\par
\indent  If $x\in L^2$, then $Sp(x)=0$. Furthermore, if $L$ is nilpotent, one has the inclusion
$$
Sp(L,E)\subseteq \{ f \in [L,L ]^{\perp}: \forall x \in L,\hbox{ } \mid f(x)\mid\le \parallel x\parallel \}. 
$$
\indent However, the spectral mapping property is ill behaved.\par
\section{The joint spectrum for solvable Lie algebras of operators}
\indent First of all we establish a proposition which will be used in the definition of 
$Sp(L,E)$.\par
\indent From now on $L$ denotes a complex finite dimensional solvable Lie algebra
and $U(L)$ its enveloping algebra.\par
\indent Let $f$ belong to $L^*$ such that $f([L,L])=0$, i.e., $f$ is a character of $L$.
Then $f$ defines a one dimensional representation of $L$ denoted by $\Bbb C(f)$. 
Let $\epsilon (f)$ be the augmentation of $U(L)$ defined by $f$:
$$
\epsilon (f)\colon U(L)\to \Bbb C(f),
$$
$$
\epsilon (f) (x)=f(x)\hskip.5cm (x\in L).
$$
\indent Let us consider the pair of spaces and maps $V(L)=(U(L)\otimes \wedge L, 
\overline{d}_{p-1})$, where $\overline{d}_{p-1}$ is the map defined by:
$$
\overline{d}_{p-1}\colon U(L)\otimes\wedge^p L\to U(L)\otimes \wedge^{p-1}L.
$$
If $p\ge 1$
\begin{align*}
\overline{d}_{p-1}&\langle x_{i_1}\cdots x_{i_p}\rangle=\sum_{k=1}^p 
(-1)^{k+1}(x_{i_k}-f(x_{i_k}))\langle x_{i_1}\wedge\cdots\wedge \hat{x}_{i_k}\wedge\cdots\wedge x_{i_p}\rangle\\
& + \sum_{1\le k<l\le p}(-1)^{k+l}\langle [x_{i_k},x_{i_l}]\wedge x_{i_1}\wedge\cdots\wedge\hat{x}_{i_k}\wedge\cdots\wedge
\hat{x}_{i_l}\wedge\cdots\wedge x_{i_p}\rangle\\
\end{align*}

where $\hat{ }$ means deletion. If $p\le 0$, we also define $\overline{d}_p=0$. \par
\begin{pro}\label{pro1}The pair of spaces and maps $V(L)$ is a chain complex. 
Furthermore, with the augmentation $\epsilon (f)$, the complex $V(L)$ is a $U(L)$-free
resolution of $\Bbb C(f)$ as a left $U(L)$-module.\end{pro}

\indent We omit the proof of Proposition \ref{pro1} because it is a straightforward generalization
of \cite[Chapter XIII, Theorem 7.1]{CE}.\par
\indent Let $L$ be as usual. From now on $E$ denotes a Banach space on which $L$ acts
as right continuous operators, i.e., $L$ is a Lie subalgebra of $L(E)$ with the opposite
product. Then, by \cite[Chapter XIII, Section 1]{CE}, $E$ is a right $U(L)$ module.\par

\indent If $f$ is a character of $L$, by Proposition \ref{pro1} and elementary homological algebra, 
the $q$-homology space of the complex $(E\otimes\wedge L, d(f))$ is
$Tor_q^{U(L)}(E, \Bbb C(f))$ $
=H_q(L,E\otimes \Bbb C(f))$.\par
\indent We now state our definition.\par
\begin{df}\label{df2} Let $L$ and $E$ be as above. The set $\{  f\in L^*: f(L^2)=0, 
H_*(L, E\otimes \Bbb C(f))\hbox{ is non-zero}\}$, is the spectrum of $L$ acting 
on $E$, and it is denoted by $Sp(L,E)$.\end{df}
\indent By Proposition \ref{pro1} and Definition \ref{df2} it is clear that if $L$ is a commutative 
algebra, then $Sp(L,E)$ reduces to the Taylor joint spectrum.\par
\indent Let's see an example. Let $(E, \parallel\hbox{} \parallel )$  be 
$(\Bbb C^2, \parallel\hbox{}\parallel_2 )$ and $a$, $b$ the opertors
$$
a=\begin{pmatrix}
               0&\frac{1}{2}\\
                \frac{1}{2}&0\\
\end{pmatrix},
\hskip1truecm
b=\begin{pmatrix}
1&1\\
-1&-1\\
\end{pmatrix}.
$$

It is easy to prove that $[b,a]=b$, and then the vector space 
$\Bbb C(b)\oplus\Bbb C(a)=L$ is a solvable Lie subalgebra of 
$L(\Bbb C^2)$.\par
\indent Using Definition \ref{df2}, a standard calculation shows that 
$Sp(L,E)= \{ f\in \Bbb C^{2*}: f(b)=0, \hbox{ } f(a)=\frac{1}{2},
\hbox{ }f(a)=\frac{-3}{2}\}$.\par
\indent Observe that, $\parallel a\parallel =\frac{1}{2}$; however,
$Sp(L,E)$ is not contained in $\{ f\in\Bbb C^{2*}: \forall x\in \Bbb C^2
\mid f(x)\mid\le \parallel x\parallel \}$.\par

\section{Fundamental properties of the sepctrum}

\indent In this section we shall prove that the most important properties 
of spectral theory are satisfied by our spectrum.\par
\vskip1cm
\begin{thm}\label{thm3}Let $L$ and $E$ be as usual. Then $Sp(L,E)$
is a conpact set of $L^*$.\end{thm}
\begin{proof}
\indent Let us consider the family of spaces and maps $(E\otimes\wedge^i L,
d_{i-1}(f))$, where $f\in L^{2^{\perp}}$ and $L^{2^{\perp}}=\{ f\in L^*: f(L^2)=0\}$.
This family is a parameterized chain complex on $L^{2^{\perp}}$.
By Taylor \cite[Theorem 2.1]{T} the set $\{ f\in L^{2^{\perp}}: (E\otimes\wedge^i L,d_{i-1}(f))
\hbox{ is exact}\}=Sp(L,E)^c$ is an open set in $L^{2^{\perp}}$.
Then, $Sp(L,E)$ is closed in $L^{2^{\perp}}$ and hence in $L^*$.\par
\indent To verify that $Sp(L,E)$ is a compact set, we consider a basis of $L^2$
and we extend it to a basis of $L$, $\{X_i\}_{1\le i\le n}$. If $K=\dim L^2$, 
$n=\dim L$ and $i\ge K+1$, then let $L_i$ be the ideal generated by $\{X_j\}_{1\le j\le n, j\neq i}$.\par
\indent Ler $f$ be a character of $L$ and represent it in the dual basis of 
$\{X_i\}_{1\le i\le n}$, $\{f_i\}_{1\le i\le n}$, $f=\sum_{i=K+1}^n
\xi_i f_i$. For each $i$, there is a positive number $r_i$ such that if $\xi_i\ge r_i$,
$$
Tor_p^{U(L)}(E,\Bbb C(f))=H_p(E\otimes\wedge^i ,d_{i-1}(f))=0\hskip.3cm\forall p.
$$
To prove our last statement, we shall construct an homotopy operator for the
chain complex $(E\otimes\wedge^p L,d_{p-1}(f))$, $(f(L^2)=0)$.\par
\indent First of all we observe that
$$
E\otimes\wedge^pL=(E\otimes \wedge^p L_i)\oplus(E\otimes \wedge^{p-1}L_i)
\wedge\langle X_i\rangle.
$$
\indent  As $L_i$ is an ideal of $L$, $d_{p-1} (E\otimes \wedge^p L_i)\subseteq E\otimes
\wedge^{p-1} L_i$. On the other hand, there is a bounded operator $L_{p-1}$
such that
$$
d_{p-1}(f)(a\wedge \langle X_i\rangle )=(d_{p-1}(f)(a))\wedge \langle X_i\rangle +
(-1)^p L_{p-1}(a), a\in E\otimes\wedge^{p-1}L_i.
$$
\indent It is easy to prove that, for each $p$, there is a basis of $\wedge^p L_i$,
$\{V_j^p\}$ $1\le j\le \dim \wedge^p L_i$, such that if we decompose
$$
E\otimes \wedge^pL_i=\bigoplus_{1\le j\le \dim \wedge^pL_i} E\langle V_j\rangle,
$$
then $L_p$ has the following form: 

\begin{align*}
&L_{p_{ij}}=\alpha^p_{ij}\hskip1.5cm \hbox{ for }i<j,\\
&L_{p_{ij}}=X_i-\xi_i+\alpha^p_{jj}, \hskip.6cm \hbox{ for }i=j\\
&L_{p_{ij}}=0,\hskip2.5cm \hbox{ for } i>j,\\\end{align*}
\noindent where $\alpha^p_{ij}\in \Bbb C$.\par

\indent Besides, let $K_p$ be a positive real number such that
$$
\bigcup_{1\le j\le \dim \wedge^p L_i} Sp(X_i+\alpha_{jj}^p)\subseteq B[0,K_p]
$$
and $N_i=\hbox{\rm max}_{0\le p\le n-1} \{ K_p\}$. Then, as $L_p$ has a triangular form, 
a standard calculation shows that $L_p$ is a topological isomorphism of Banach
spaces if $\xi_i\ge N_i$.\par
\indent Outside $B[0,N_i]$ we construct our homotopy operator
$$
Sp\colon E\otimes\wedge^p L\to E\otimes \wedge^{p+1}L,
$$
$$
Sp\mid_{E\otimes \wedge^{p-1}L_i\wedge\langle X_i\rangle}\equiv 0,
$$
$$
Sp\colon E\otimes\wedge^p L_i\to E\otimes \wedge^p L_i\wedge \langle X_i\rangle,
$$
$$
Sp=(-1)^{p+1}L_p^{-1}\wedge\langle X_i\rangle.
$$
\indent From the definition of $L_p$ we have the following identity:
$$
(-1)^{p+2}S_{p-1}d_{p-1}(f)L_p=d_{p-1}(f)\wedge\langle X_i\rangle.
$$
\indent The above identity and a standard calculation prove that $Sp$ is a 
homotopy operator, i.e., $d_pS_p +S_{p-1}d_{p-1}=I$. Therefore, $Sp(L,E)$
is a compact set.\end{proof}

\begin{thm}\label{thm4} (Projection Property) Let $L$ and $E$ be as usual, and let $I$
be an ideal of $L$. Let $\pi$ be the projection map from $L^*$ onto $I^*$, then
$$
Sp(I,E)=\pi (Sp(L,E)).
$$\end{thm}
\begin{proof}
\indent By \cite[Section 5.3, Corollaire 3]{B}, there is a Jordan H\"older sequence of $L$ such that $I$ is one of 
its terms. Then, by means of an induction argument, we can assume $\dim(L/I)=1$.\par
\indent Let's consider the connected simply connected complex Lie group $G(L)$ such 
that its algebra is $L$, \cite[Part II-Lie Groups, Chapter 5]{S}.\par
\indent Let $Ad^*$ be the coadjoint representation of $G(L)$ in $L^*$: $Ad^*(g)(f)=fAd(g^{-1})$,
where $g\in G(L)$, $f\in L^*$ and $Ad$ is the adjoint representation of $G(L)$ in $L$.\par
\indent Let $f$ belong to $Sp(I,E)$. Then, as $I$ is an ideal of $L$, by \cite[Theorem 2.13.4]{V},
$Ad^*(g)(f)$ belongs to $I^*$. Besides, it is a character of $I$.
Then, one can restrict the coadjoint action of $G(L)$ to $I^*$.
Moreover, $Sp(I,E)$ is invariant under the coadjoint action of $G(L)$ in
$I^*$, i.e., if $f\in Sp(I,E)$, then
$Ad^*(g)(f)\in Sp(I,E)$, $\forall g\in G(L)$.\par
\indent To this end, it is enough to prove that
$$
(I)\hskip2truecm Tor_*^{U(I)}(E,\Bbb C(f))\cong Tor_*^{U(I)}(E,\Bbb C (h))
$$
where $h=Ad^*(g)(f)$, $g\in G(L)$.\par
\indent Let $\Gamma$ be the ring $U(I)$ and $\varphi$ the ring morphism
$$
\varphi=U(Adg)\colon U(I)\to U(I).
$$
\indent Let's consider the augmentation modules $(\Bbb C(f), E(f))$ and 
$(\Bbb C(h), E(h))$.\par
\indent Then, a standard calculation shows that the hypothesis of \cite[Chapter VIII, Theorem 3.1]{CE}
 are satisfied, which implies $(I)$.\par
\indent Thus, if $f\in Sp(I,E)$, the orbit $G(L)\cdot f\subseteq Sp(I,E)$. However, 
$Sp(I,E)$ is a compact set of $I^*$.\par
\indent As the only bounded orbits for an action of a complex connected Lie 
group on a vector space are points, then $G(L)\cdot f=f$.\par
\indent Let $\overline{f}$ be an extension of $f$ to $L^*$, and consider $\alpha=G(L)\cdot 
\overline{f}$, the orbit of $\overline{f}$ under the coadjoint action of $G(L)$ in $L^*$.\par
\indent As $G(L)\cdot f=f$, then as an analytic manifold\par

\hskip1.7cm (II)\hskip2.5truecm  $\dim\alpha \le 1$.\par

\indent Now suppose $\overline{f}$ is not a character of $L$, i.e.,
$\overline{f} (L^2)\ne 0$.\par
\indent Let $L^{\perp}$ be the following set: $L^{\perp}=\{ x\in L: \overline{f} ([X,L])=0\}$,
 and let $n$ be the dimension of $L$.\par
\indent As $I$ is an ideal of dimension $n-1$, $f(I^2)=0$ and $f(L^2)\neq 0$, by  \cite[Section 5.3, Corollaire 3]{B},
\cite[Chapitre IV, Section 4.1, Proposition 4.1.1]{BCV} and by \cite[Chapter 1, 1.2.8]{D}, we have $L^{\perp}\subset I$ and $\dim L^{\perp}=n-2$.\par
\indent Let's consider the analytic subgroup of $G(L)$ such that its Lie algebra is 
$L^{\perp}$.\par
\indent As the Lie algebra of the subgroup $G(L)_{\overline{f}}=\{ g\in G(L): Ad^*(g)(\overline{f})
=\overline{f}\}$ is $L^{\perp}$, the connected component of the identity of 
$G(L)_{\overline{f}}$ is $G(L^{\perp})$.\par
\indent Howecwer, by \cite[Lemma 2.9.2, Theorem 2.9.7]{V} $\alpha=G(L)\cdot\overline{f}$ satisfies the following
properties: $\alpha\cong G(L)/G(L)_{\overline{f}}$, and $\dim \alpha=\dim G(L)-\dim 
G(L)_{\overline{f}}=\dim G(L)-\dim G (L^{\perp})=\dim L-dim L^{\perp}=2$, which contradicts (II).
\par
\indent Then, $\overline{f}$ is a character of $L$.\par
\indent Thus, any extension $\overline{f}$ of and $f$ in $Sp(L,E)$ is a character of $L$.
\par
\indent However, as in \cite{T}, there is a short exact sequence of complexes
$$
0\to (\wedge^* I\otimes E , d(f))\to(\wedge^* L\otimes E , d(\overline{f}))\to
(\wedge^* I\otimes E , d(f))\to 0.
$$
\indent As $U(I)$ is a subring with unit of $U(L)$ and the complex involved in Definition \ref{df2}
differs from the one of  \cite{T} by a constant term, Taylor's joint argument of \cite[Lemma 3.1]{T}
still applies and then $Sp(I,E)=\pi (Sp(L,E)$.\end{proof}
\indent As a consequence of Theorem \ref{thm4} we have the following result.\par
\begin{thm}\label{thm5}Let $L$ and $E$ be as usual. Then $Sp(L,E)$ is not void.
\end{thm}
\section{Some Consequences}

\indent In this section we shall see some consequences of the main theorems.\par
\indent Let $E$ be a Banach space and $L$ a complex finite dimensional solvable
Lie algebra acting on $E$ as bounded operators.\par
\indent One of the well known properties of Taylor joint spectrum for an $n$-tuple of 
m.c.o. acting on $E$ is $Sp(a,E)\subseteq \Pi B[0,\parallel a_i\parallel]$. In the 
noncommutative case, as we have seen in section 2, this property fails.\par
\indent However, if the Lie algebra is nilpotent, it is still true.\par
\begin{pro}\label{pro6} Let $L$ be a nilpotent Lie algebra which acts as bounded 
operators on a Banach space $E$.\par
\indent Then, $Sp(L,E)\subseteq \{f\in L^*:\hbox{ }\mid f(x)\mid\le\parallel x\parallel
\hbox{, }x\in L \}$.\end{pro}
\begin{proof}
\indent We proceed by induction on $\dim L$. If $\dim L=1$, we have nothing to verify.\par
\indent We suppose that the proposition holds for every nilpotent Lie algebra $L'$ such that
$\dim L'<n$.\par
\indent If $\dim L=n$, by \cite[Section 4.1, Proposition 1]{B}, there is a Jordan H\"older sequence $S=(L_i)_{0\le i\le n}$, such
that $[L,L_i]\subseteq L_{i-1}$.\par
\indent Let $\{X_i\}_{1\le i\le n}$ be a basis of $L$ such that $\{X_j\}_{1\le j\le i}$ generates
$L_i$.\par
\indent Let $L'_{n-1}$ be the vector subspace generated by $\{ X_i\}_{1\le i\le n, i\ne n-1}$.
As $[L,L_{n-1}']\subseteq L_{n-2}\subseteq L_{n-1}'$, $L_{n-1}'$ is an ideal . Besides,
$L_{n-1}+L_{n-1}'=L$.\par
\indent Then, by means of Theorem \ref{thm4} and the inductive hypothesis, we complete the inductive 
argument and the proposition.\end{proof}
\indent Now, we deal with some consequences of the projection property.\par
\begin{pro}\label{pro7}Let $L$ and $E$ be as usual. If $I$ is an ideal contained in $L^2$, 
then $Sp(I,E)=0$. In particular, $Sp(L^2,E)=0$.
\end{pro} 
\begin{proof}
\indent By the projection property, $Sp(I,E)=\pi (Sp(L,E))$, where $\pi$ is the projection 
from $L^*$ onto $I^*$. However, as $Sp(L,E)$ is a subset of characters of $L$, $f\mid_I=0$,
if $I\subseteq L^2$.\end{proof}
\begin{pro}\label{pro8}Let $L$ and $E$  be as in Proposition \ref{pro6}. If $Sp(L,E)=0$, then
$Sp(x)=0$, $\forall x\in L$. \end{pro}
\begin{proof}
\indent By means of and inductive argument, the ideals $L_{n-1}$, $L_{n-1}'$ of 
Proposition \ref{pro6} and Theorem \ref{thm4}, we conclude the proof.\end{proof}
\begin{pro} \label{pro9}Let $L$ and $E$ be as usual. Then, if $x\in L^2$, 
$Sp(x)=0$.\end{pro}
 \begin{proof}
\indent First of all, recall that if $L$ is a solvable Lie algebra, then $L^2$ is a nilpotent
Lie algebra. Therefore, by Proposition \ref{pro7}, $Sp(L^2,E)=0$, and by Proposition \ref{pro8}, $Sp(x)=0$ $\forall x\in L^2$.
\end{proof}
\begin{rema}\rm
\indent Note that the example of section 2 shows that the projection property fails for subspaces
which are not ideals (take $I=\langle x \rangle$). Clearly this implies that the spectral
mapping theorem also fails in the noncommutative case.
\end{rema}

\bibliographystyle{amsplain}

\vskip.3truecm
\noindent Enrico Boasso\par
\noindent E-mail address: enrico\_odisseo@yahoo.it

\end{document}